\documentclass[11pt]{article}
\usepackage{amsmath}
\usepackage{amsfonts}
\usepackage{amssymb,enumerate}
\usepackage{amsthm}
\usepackage{comment}
\usepackage{natbib}
\usepackage{setspace}
\usepackage{footnote}
\usepackage{algorithm}
\usepackage{algorithmic}
\usepackage{graphicx}
\usepackage{pstricks}
\usepackage{pst-plot}
\usepackage{pstricks-add} 
\usepackage{color,pstricks} 
\usepackage{colortbl} 
\usepackage{caption} 
\usepackage{subcaption}
\usepackage{epsf}
\usepackage{color}
\usepackage{booktabs}
\usepackage{multirow}
\usepackage{hyperref}
\usepackage{bm}
\usepackage{comment}
\hypersetup{
    bookmarks=true,         
    unicode=false,          
    pdftoolbar=true,        
    pdfmenubar=true,        
    pdffitwindow=true,      
    pdftitle={Relax_stability},    
    colorlinks=true,       
    linkcolor=blue,          
    citecolor=blue,        
    filecolor=magenta,      
    urlcolor=cyan           
}
\newtheorem{theorem}{Theorem}
\newtheorem{assumption}{Assumption}
\newtheorem{lemma}{Lemma}

\usepackage[top=1.1in, bottom=1.1in, left=1.1in, right=1.1in]{geometry}
\newcommand{\orc}{Operations Research Center, Massachusetts Institute of Technology, Cambridge, Massachusetts 02139}

\title{Should We Relax Stability in Matching Markets?}
\author{Dimitris Bertsimas \thanks{Sloan School of Management, Massachusetts Institute of Technology, Cambridge, Massachusetts 02139} \and Carol Gao \thanks{\orc} 
  }\date{\today}

\begin{document}
\setlength{\abovedisplayskip}{8pt}
\setlength{\belowdisplayskip}{8pt}

\maketitle

\begin{abstract}
    Centralized assignment markets have historically relied on Deferred-Acceptance (DA) algorithms, which do not  incorporate multiple objectives into the assignment. In this work, we propose an optimization-based many-to-one assignment algorithm that explores the trade-offs between minimizing the number of blocking pairs in the match and other important objectives. In order to scale to high-dimensional problems, we develop an algorithm using inverse optimization to obtain the optimal cost vector that implicitly optimizes for stability. This is empirically tested on two application areas for which the DA algorithm is widely used - school assignment and medical residency match. Computational tests on a simulated Boston Public Schools (BPS) match show that this method effectively reduces transportation cost and increases number of students receiving an offer in the first round of match at the expense of a small percentage of blocking pairs. Similar improvement in the number of couples matched to the same location is observed in a synthetic residency match.
\end{abstract}

\newpage
\section{Introduction}
School assignment in urban districts and hospital-resident matching are two widely studied assignment problems that use a centralized assignment mechanism and employ the Deferred-Acceptance (DA) algorithm \citep{roth2008}. 
The DA algorithm is guaranteed to generate a stable matching, namely there exists no blocking pair. In the school choice setting, this means that there exists no pair of student and school that prefer each other to their current match. 

However, the literature is characterized by an almost religious devotion to stability. 
The DA matching mechanism has two major limitations. First, it focuses on stability as the sole objective and fails to incorporate other important objectives. The assignment is either optimal from the students' perspective using the student-proposing algorithm or optimal from the schools' perspective using the school-proposing algorithm. 
Second, the DA algorithm has little variability. Specifically,  \cite{roth1986} proves that in all stable solutions, the set of applicants that receive an offer is the same. Therefore, there is little variability for adjusting the match rate while maintaining stability. This potentially leads to a significant number of unmatched applicants and vacant positions that could be matched to each other.

In the context of school choice, in large urban school districts, transportation cost constitutes a major component of district budget. For example, in Boston Public Schools (BPS), transportation cost accounts for more than 10\% of the district's overall spending, making it the second largest spending category after teachers and paraprofessionals salary. Moreover, the cost has increased by more than 24 million (17\%) since 2018 \citep{bps_review_2022}. 
Even when considering transportation cost as the sole objective, in an experiment of more than 16,000 students and 90 schools, the total transportation cost differs no more than 0.13\% than among all possible stable solutions \citep{papalexopoulos2022}. 
Moreover, the unmatched students will have to resubmit preferences and participate in an additional round of match that allocates them to the unfilled positions, which eventually violates stability if considered with the assignment in the initial round. Furthermore, conditional on receiving an offer, the probability of reapplication varies less than 10\% between receiving a first choice and receiving the last choice in rank order list \citep{narita2018}. Hence, in order to boost enrollment in public school districts, it is important to raise the number of offers made in the initial assignment process and minimize the number of wasted capacities. 


In medical residency matches, a long-lasting challenge is the existence of couples. Couples comprise a significant proportion of applicants in residency match and accommodating the preference of matching couples to the same hospital has been a long-lasting issue in residency matching markets. However, the existence of couples is difficult to formulate within the DA algorithm, particularly because there is no universal definition of stability when individuals submit single preferences and couples submit joint preferences. \cite{kojima2013} and \cite{mcdermid2010} have defined stable matching when couples submit joint preferences.  \cite{biro2013} illustrate that there is no unique definition of stability when there are joint preferences.
Additionally, when couples submit joint preferences over pairs of programs, there may exist no stable assignment \citep{roth2008}. 
These characteristics make DA an ill-suited algorithm to model the residency problem in the existence of couples. 

\subsection{Literature}
This work is primarily related to the applications of optimization techniques and DA algorithm in the context of school choice.
The student-proposing DA algorithm has been widely employed in large urban districts including Boston and New York City \citep{abdulkadirouglu2005nyc, abdulkadirouglu2005bos}, due to its desirable property of stability. 
\cite{baiou2000} describe the convex hull of stable assignments using a set of linear constraints. \cite{delorme2019} use pre-processing heuristics to scale an mixed integer optimization formulation of the stable assignment problem.
Recent literature also uses optimization techniques to incorporate other objectives while maintaining stability in designing school assignment mechanisms \citep{ashlagi2014, shi2015}. 
In particular, \cite{ashlagi2014} implement a correlated lottery to maximize community cohesion. \cite{bodoh2020} considers a combination of welfare and diversity.
\cite{papalexopoulos2022} develops a custom pre-solve algorithm for the fully stable case that scales the mixed integer optimization  (MIO) to problems of the size of a realistic BPS match. In particular, when no blocking pair is allowed, one could leverage the Rural Hospitals Theorem \citep{roth1986} which specifies that the set of students matched and seats filled at each school are the same in all stable solutions, and any school that is under capacity is matched to the same set of students. Therefore, the MIO problem can be reduced significantly by identifying matches that are fixed in all stable solutions using these rules. 

Another strand of literature studies the assignment mechanism of residency match and particularly the case with couples.
Related to the objective of matching couples to the same location, multiple definitions of stability have been given in a residency match where couples submit joint preferences over pairs of hospitals \citep{kojima2013, mcdermid2010}. They show that stable matching does not always exist when there are couples. 
More recently, \cite{manlove2017} find assignments that minimize the number of blocking pairs and show that such problem is NP-hard and difficult to approximate.

\subsection{Contributions}
We propose a novel assignment algorithm using inverse optimization that incorporates multiple objectives and allows for adjustments on the weights between stability and other important objectives. Our contributions are as follows: 
\begin{itemize}
    \item {High benefit of relaxing stability:} In the context of school choice, we present a trade-off curve on the number of blocking pair created and average travel distance reduced. Policy-makers may choose the weights they would like to place on stability and other objectives to achieve the metrics appropriate for their school districts. Other objectives, such as racial integration, can also be easily incorporated into the objective as needed.  
    \item {{Versatility:}} In a simulated match based on the BPS transportation challenge, we demonstrate up to more than 20\% reduction in average travel distance of students with as low as 5\% of blocking pairs. Similarly, allowing for 5\% of blocking pairs in a residency match of the same problem size leads to a roughly 10\% improvement to the number of couples matched to the same location, as compared to the DA solution.
    \item {{Tractability and Scalability:}}Directly formulating blocking pairs using binary variables and linear constraints leads to 100 million constraints and is challenging to scale to realistic problem sizes when stability is relaxed and the Rural Hospitals Theorem cannot be leveraged \citep{papalexopoulos2022}. However, our method can be scaled to large-scale problems with more than 16,000 students and 90 schools, which constitutes more than 80,000 possible pairs\textemdash the realistic match size at BPS. This algorithm produces an assignment in seconds. 
    \item {{Flexibility:}} Only a limited number of students can be matched without creating blocking pairs. Consequently, after the DA process, many students remain unmatched and some school seats are left vacant.
    In reality, these students and positions are reallocated in an additional round of match, which will eventually lead to blocking pairs. In contrast, our algorithm takes the minimum number of students matched as an input and can fill these positions more optimally in the initial round with a small number of blocking pairs. This gives policy-makers the flexibility to decide how many students they would like to be able to match in the main round without being constrained by the DA algorithm. 
\end{itemize}

\section{Optimization Methodology}
Both the school assignment and the residency match problems consist of a set of applicants $\mathcal{I}$ and a set of programs (schools) $\mathcal{J}$. Each applicant $i$ can at most be matched to one program and each program $j$ cannot admit more applicants than its capacity $C_j$.
Let $p_{i,j}$ denote the utility of applicant $i$ from being matched to program $j$. Similarly, let $q_{i,j}$ denote the utility of program $j$ from admitting applicant $i$.
Each applicant $i$ has a preference set of all the valid programs that they would prefer over being unmatched, denoted by $\mathcal{P}(i) = \{j \in \mathcal{J}: p_{i,j} > 0 \}$. Likewise, let $\mathcal{Q}(j) = \{i \in \mathcal{I}: q_{i,j} > 0 \}$ denote the set of applicants that are admissible to program $j$. Define the set of all valid pairs to be $\mathcal{S} = \{(i,j) \in \mathcal{I} \times \mathcal{J}: p_{i,j} > 0, q_{i,j} > 0\}$.
Let $\mathcal P_>(i,j) = \{k \in \mathcal{J}: p_{i,k} > p_{i,j}\}$ represent the set of programs that applicant $i$ prefers to program $j$, and $\mathcal Q_>(i,j) = \{k \in \mathcal{I}: q_{k,j} > q_{i,j}\}$ represent the set of applicants that program $j$ prefers to applicant $i$.\footnote{Similarly, $\mathcal P_<(i,j) = \{k \in \mathcal{J}: p_{i,k} < p_{i,j}\}$ is the set of programs less preferred to program $j$ for applicant $i$. $\mathcal Q_<(i,j) = \{k \in \mathcal{I}: q_{k,j} < q_{i,j}\}$ is the set of applicants less preferred to applicant $i$ for program $j$. }

This assignment problem can be formulated using the following set of binary variables 
\begin{align*}
    x_{i,j} \in  \{0,1\}  \quad \forall (i,j) \in \mathcal{S}, \quad \text{ where $x_{i,j}$ =1, if applicant $i$ is matched to program } j.
\end{align*}
Applicant $i$ and program $j$ constitute {\bf  a blocking pair} if \\
{\bf  (a1)} applicant $i$ is unmatched {\bf  or} ({\bf  a2)} applicant $i$ is matched to school $k$ but they prefer $j$ to $k$ (i.e.,  $x_{i,k} = 1$ for $j \in \mathcal{P}_>(i,k))$,\\
 {\bf  and}\\
  {\bf  (b1)} program $j$ is under capacity {\bf  or (b2)} program $j$ admits any applicant $k$ but they prefer applicant $i$ to $k$ (i.e.,  there exists a applicant $k$ such that $x_{k,j} = 1$ and $i \in \mathcal{Q}_>(k,j) )$. 

The definition of blocking pairs can be directly formulated by introducing the following binary variables:   
\[
\begin{array}{@{}l l@{}}
    z_i  \in  \{0,1\}  \quad \forall i \in \mathcal{I}, &\quad \text{ where $z_i=1$ if applicant $i$ is unmatched (i.e., ({\bf  a1}) is true)}, \\
    y_j  \in  \{0,1\}  \quad \forall j \in \mathcal{J}, &\quad \text{ where $y_j=1$ if program $j$ is under capacity (i.e., ({\bf  b1}) is true)}, \\
    r_{i,j} \in  \{0,1\}  \quad \forall (i,j) \in \mathcal{S}, &\quad \text{ where $r_{i,j} =1$ if applicant $i$ is unhappy with respect to program $j$}\\
    & \quad \text{ (i.e., ({\bf  a1 or a2}) is true)} ,  \\
    w_{i,j} \in  \{0,1\}  \quad \forall (i,j) \in \mathcal{S}, &\quad \text{ where $r_{i,j} =1$ if program $j$ is unhappy with respect to applicant $i$}  \\
    &\quad \text{ (i.e., ({\bf  b1 or b2}) is true)}, \\
    a_{i,j} \in \{0,1\}  \quad \forall (i,j) \in \mathcal{S}, &\quad \text{ where $a_{i,j}=1$ if applicant $i$ and program $j$ constitute a blocking pair}.
\end{array}
\]
The following set of linear constraints effectively capture the space of allowing no more than $B$ number of blocking pairs: 
\begin{align}
    z_i = 1 - \sum_{j\in \mathcal{P}(i)}x_{i,j}, & \quad \forall i \in \mathcal{I}, \label{const:stu_unmatched} \\
    y_j \le C_j - \sum_{i\in \mathcal{Q}(j)} x_{i,j} \le C_j y_j, & \quad \forall j \in \mathcal{J}, \label{const:sch_under_capacity} \\ 
    r_{i,j} = z_i + \sum_{k \in \mathcal{P}_{<}(i,j)}x_{i,k}, & \quad \forall (i,j) \in \mathcal{S}, \label{const:stu_unhappy} \\
    w_{i,j} \ge y_i , & \quad \forall (i,j) \in \mathcal{S}, \label{const:sch_unhappy1} \\
    w_{i,j} \ge x_{k,j} , & \quad \forall (i,j) \in \mathcal{S}, k \in \mathcal{Q}_{<}(i,j),  \label{const:sch_unhappy2}\\
    w_{i,j} \le y_j + \sum_{k \in \mathcal{Q}_{<}(i,j)}x_{k,j}, & \quad \forall (i,j) \in \mathcal{S},\label{const:sch_unhappy3} \\ 
    r_{i,j} + w_{i,j} \le 1 + a_{i,j}, & \quad \forall (i,j) \in \mathcal{S}, \label{const:blocking} \\
    \sum_{(i,j) \in \mathcal{S}} a_{i,j} \le B. \label{const:blocking_number}
\end{align}
Specifically, constraint \eqref{const:stu_unmatched} indicates that applicant $i$ is unmatched if they are not assigned to any of the programs on their preference list. Constraint \eqref{const:sch_under_capacity} implies that $y_j = 0$ (i.e., program $j$ is not under capacity) only if the number of applicants assigned to it is exactly its capacity.
Constraint \eqref{const:stu_unhappy} directly formulates conditions {\bf  (a1)} and {\bf  (a2)}, ensuring that $r_{i,j} =1$ if either applicant $i$ is unmatched or they are assigned to some program $k$ that is less preferred to them than program $j$. Analogously, constraints \eqref{const:sch_unhappy1} to \eqref{const:sch_unhappy3} formulate conditions {\bf  (b1)} and {\bf  (b2)}, ensuring that $w_{i,j} =1$ if either program $j$ is under capacity or they have admitted an applicant $k$ that is less preferred to them than applicant $i$.
Finally, by the definition of blocking pair, $(i,j)$ constitutes a blocking pair (i.e., $a_{i,j}$ = 1) if both $r_{i,j} = 1$ and $w_{i,j} =1$, as reflected in constraint \eqref{const:blocking}. The number of blocking pairs can then be explicitly constrained to be no more than $B$ in constraint \eqref{const:blocking_number}. 

As noted in \cite{papalexopoulos2022}, this exact method leads to roughly 300,000 binary variables and 100 million constraints when there are 16,255 students and 92 schools, making it not scalable in an assignment problem of the size of BPS. 
Therefore, in the next section, we present a novel optimization algorithm that optimizes for stability and other objectives of interest without directly formulating the blocking pair definition. In particular, we use inverse optimization to extract a cost vector that makes the known stable solution optimal. We tailor this approach to school choice matches in which the objective in addition to stability is minimizing transportation cost and medical residency matches in which the objective is to maximize couples matched to the same location. We refer to this formulation as the {\bf  inverse method}.


\subsection{School Choice}
The school assignment problem is defined by a many-to-one mapping from the set of students $\mathcal{I}$ to the set of schools $\mathcal{J}$. When the sole objective is stability, the problem that minimizes the number of blocking pairs is given by 
\begin{align}
    \min_{\mathbf x} \quad &  \sum_{(i,j) \in \mathcal{S}} b_{i,j} x_{i,j} \\ 
    \text{s.t.} \quad & \sum_{j \in \mathcal{P}(i)} x_{i,j} \le 1, \quad \forall i \in \mathcal{I} ,\label{const:unique_sch}\\
    & \sum_{i \in \mathcal{Q}(j)} x_{i,j} \le C_j, \quad \forall j \in \mathcal{J}, \label{const:capacity} \\ 
    & \sum_{(i,j) \in \mathcal{S}} x_{i,j} \ge N, \label{const:min_matched} \\
    &x_{i,j} \in \{0,1\}, \quad \forall (i,j) \in \mathcal{S},
\end{align}
where $\sum_{(i,j) \in \mathcal{S}} b_{i,j} x_{i,j}$ is the objective that minimizes the number of blocking pairs, but $\mathbf{b}$ is unknown. Constraint \eqref{const:unique_sch} ensures each student is matched to at most one school and Constraint \eqref{const:capacity} specifies that each school does not admit more students than its capacity.

We approximate this integer formulation as a linear optimization (LO) problem by relaxing integrality to $x_{i,j} \in [0,1]$. Using strong duality of this relaxed LP, we derive an inverse optimization problem characterized by the following dual variables: 
\begin{align*}
    & b_{i,j} \in \mathbb{R} ,\quad \forall (i,j) \in\mathcal{S}, \\
    & u_i \le 0, \quad \forall i \in \mathcal{I}, \\
    & v_j \le 0, \quad \forall j \in \mathcal{J}, \\
    & w \ge 0. 
\end{align*}
The inverse optimization problem is given by 
\begin{align}
    \min_{\mathbf {b, u, v, w}} \quad & \mathbf b^\top \mathbf x + \lambda||\mathbf b - \mathbf{\overline b} ||_2 \label{const:inv_objective} \\ 
    \text{s.t.} \quad & u_i + v_j + w \le  b_{i,j}, \quad \forall (i,j) \in \mathcal{S}, \label{const:dual} \\
    & (b_{i,j} - u_i + v_j -w)x_{i,j}^{\text{DA}} = 0, \quad \forall (i,j) \in \mathcal{S}, \label{const:comp_slack}\\ 
    & \sum_{(i,j) \in \mathcal{S}} b_{i,j} x^{\text{DA}}_{i,j} = \sum_{i \in \mathcal{I}} u_i + \sum_{j \in \mathcal{J}} C_j v_j + Nw, \label{const:equal_obj} \\
    & \sum_{(i,j) \in \mathcal{S}} b_{i,j} = 1, \label{const:normalization}
\end{align}
where $\mathbf{\overline b}$ is a initiating vector and $\lambda||\mathbf b - \mathbf{\overline b} ||_2$ is a regularization term \citep{ahuja2001inverse, chan2025inverse}.
Constraint \eqref{const:dual} corresponds to the dual constraints. Constraint \eqref{const:comp_slack} ensures complementary slackness. Constraint \eqref{const:equal_obj} ensures that the primal objective equals the dual objective at optimality. 
We also eliminate trivial solutions of $b_{i,j} = 0$ for all $(i,j)$ by introducing a normalization constraint \eqref{const:normalization}. Finally, to avoid having trivial solutions of only having $b_{i,j} = 1$ for one pair $(i,j)$, we initiate with a cost vector $\mathbf{\overline b}$ \citep{ahuja2001inverse}.

Solving the inverse problem yields an optimal cost vector $\mathbf{ b^*}$ that makes the DA solution $\mathbf{x}^{\text{DA}}$ optimal. Let $d_{i,j}$ denote the distance for student $i$ to travel from home to school $j$. 
We are then ready to solve a multi-objective assignment problem that reaches a target travel objective $T$ while minimizing the number of blocking pairs, given by 
\begin{align}
    \min_{{\bf x}} \quad &\lambda_1 \sum_{(i,j) \in \mathcal{S}}b^*_{i,j}x_{i,j} \underbrace{- \lambda_2 \sum_{(i,j) \in \mathcal{S}} \frac{|\mathcal{P}_<(i,j)|}{|P(i)|}x_{i,j}}_{\text{maximize student preferences}} \underbrace{- \lambda_3 \sum_{(i,j) \in \mathcal{S}} \frac{|\mathcal{Q}_<(i,j)|}{|Q(i)|}x_{i,j}}_{\text{maximize school preferences}} \\
    \text{s.t.} \quad & \sum_{j \in \mathcal{P}(i)} x_{i,j} \le 1, \quad \forall i \in \mathcal{I}, \\
    & \sum_{i \in \mathcal{Q}(j)} x_{i,j} \le C_j, \quad \forall j \in \mathcal{J}, \\ 
    & \sum_{(i,j) \in \mathcal{S}} x_{i,j} \ge N, \\ 
    & \sum_{(i,j) \in \mathcal{S}} d_{i,j}x_{i,j} \le T,\label{const:strict_travel} \\
    &x_{i,j} \in \{0,1\}, \quad \forall (i,j) \in \mathcal{S},
\end{align}
where $T$ is the target total travel distance that can be pre-specified based on budget constraint. 
Recall that $|\mathcal{P}_<(i,j)|$ denotes the number of schools that student $i$ prefers less to school $j$. Note that the set $\mathcal{P}_<(i,j)$ expands as $j$ moves towards the top of the preference list. Meanwhile, $|\mathcal{P}(i)|$ represents the number of all schools ranked by student $i$. 
Therefore, a larger $\frac{|\mathcal{P}_<(i,j)|}{|P(i)|}$ indicates school $j$ is ranked higher in student $i$'s preference list. 
Specifically, if $j$ is ranked last among all of $i$'s ranked schools, then  $|\mathcal{P}_<(i,j)| = 0$ and $\frac{|\mathcal{P}_<(i,j)|}{|\mathcal{P}(i)|} =0$. Conversely, if $j$ is ranked first, then every other school is less-preferred than $j$, then $\frac{|\mathcal{P}_<(i,j)|}{|P(i)|} = \frac{|P(i)|-1}{|P(i)|}$. Therefore, $\frac{|\mathcal{P}_<(i,j)|}{|P(i)|} \in [0,1)$ and by maximizing $\frac{|\mathcal{P}_<(i,j)|}{|P(i)|}$, we are maximizing the students' preferences by matching them to the highest ranked schools possible. The same argument can be made about the school $j$'s preferences by maximizing $\frac{|\mathcal{Q}_<(i,j)|}{|Q(i)|}$. Altogether, these two objectives implicitly ensure that when the solution deviates from the original stable solution, the deviation is in the direction that maximizes both the students' and the schools' preferences, therefore incur minimal blocking pairs. 
The full algorithm is given in Algorithm \ref{alg:inv}.

\subsection{Residency Match}
Analogous to the school choice problem, the hospital-resident matching problem is defined by a set of residents $\mathcal I$ and a set of hospitals $\mathcal J$. A subset of residents $\mathcal R \subset \mathcal I$ are applying with a partner. $D$ is a one-to-one mapping from $\mathcal{R}$ to $\mathcal{R}$ where $D(i) \in \mathcal{R}$ is the partner of resident $i$ for all $i \in \mathcal{R}$. 
We assume that a set of couples $(i, D(i))$ prefer being matched to the same hospital. 
For tractability, we also assume that a set of couple has the same set of hospitals that they find valid, namely $\mathcal P(i) = \mathcal P(D(i))$ for all $i \in \mathcal{R}$, but they may have different preferences over the hospitals within this set, namely for some $j \in \mathcal P(i)$, it is possible that $\mathcal P_>(i,j) \neq \mathcal P_>(D(i),j)$.
To this end, for each set of couple $(i, D(i))$, we model the couple preference with 
\begin{align}
    \sum_{j \in \mathcal P(i)} x_{i,j}x_{D(i),j}.
\end{align}
This expression is equal to 1 if and only if $i$ and $D(i)$ are matched to the same hospital $j$. 
Note that this objective can easily be linearized by introducing auxiliary variables 
$$
y_{i,D(i),j} \in \{0,1\} \quad \forall i \in \mathcal{R}, j \in \mathcal{P}(i), 
$$
which denote whether applicant $i$ and their partner $D(i)$ are both matched to hospital $j$.
The objective for maximizing couple preferences can then be modeled as matching at least $T$ couples to the same location can be modeled as 
\begin{align} 
    & \sum_{i \in \mathcal{R}} \sum_{j \in \mathcal{P}(i)} y_{i,D(i),j} \ge T \\
    &y_{i,D(i),j} \le x_{i,j}, \quad \forall i \in \mathcal{R}, j \in \mathcal{P}(i), \\
    &y_{i,D(i),j} \le x_{D(i),j}, \quad \forall i \in \mathcal{R}, j \in \mathcal{P}(i), \\
    & y_{i,D(i),j} \ge x_{i,j} +  x_{D(i),j} - 1,\quad \forall i \in \mathcal{R}, j \in \mathcal{P}(i). 
\end{align}
Then, the multi-objective optimization that minimizes the number of blocking pairs while achieving the target number of couples matched to the same location is given by 
\begin{align}
    \min_{\mathbf x} \quad & \lambda_1  \sum_{(i,j) \in \mathcal{S}} b^*_{i,j} x_{i,j} \underbrace{- \lambda_2 \sum_{(i,j) \in \mathcal{S}} \frac{|\mathcal{P}_<(i,j)|}{|P(i)|}x_{i,j}}_{\text{maximize applicant preferences}} \underbrace{- \lambda_3 \sum_{(i,j) \in \mathcal{S}} \frac{|\mathcal{Q}_<(i,j)|}{|Q(i)|}x_{i,j}}_{\text{maximize hospital preferences}} \label{max_couple} \\ 
    \text{s.t.} \quad & \sum_{j \in \mathcal{P}(i)} x_{i,j} \le 1, \quad \forall i \in \mathcal{I} \\
    & \sum_{i \in \mathcal{Q}(j)} x_{i,j} \le C_j, \quad \forall j \in \mathcal{J} \\ 
    & \sum_{(i,j) \in \mathcal{S}} x_{i,j} \ge N, \\ 
    & \sum_{i \in \mathcal{R}} \sum_{j \in \mathcal{P}(i)} x_{i,j}x_{D(i), j} \ge T, \label{const:couple}\\
    &x_{i,j} \in \{0,1\}, \quad \forall (i,j) \in \mathcal{S},
\end{align}
where $\sum_{i \in \mathcal{R}} \sum_{j \in \mathcal{P}(i)} x_{i,j}x_{D(i), j}$ is the number of residents being matched to the same location as their partner. Constraint \eqref{const:couple} ensures that at least $T$ couples are matched to the same location.

\begin{algorithm}[H] 
\makeatletter
\makeatother
\begin{algorithmic}
\STATE{\textbf{Input:} applicants  $\mathcal{I}$ and programs $  \mathcal{J}$}
\STATE{\textbf{Parameters:} \\
\quad $\mathcal{P}(i):$ preferences of applicant $i$ \\
\quad $\mathcal{Q}(j):$ preferences of program $j$ \\
\quad $C_j:$ capacity of program $j$ \\
\quad $T:$ target for other objective \\
\quad $(\lambda_1,\lambda_2, \lambda_3):$ weights on minimizing distance to stable solution, maximizing applicant preferences, and maximizing program preferences  \\
\quad $N:$ minimum number of applicants to be matched}
\STATE{\textbf{Output:} assignment solution $\mathbf{x^*}$}
\STATE{\emph{Step 1:} Run Deferred-Acceptance algorithm to obtain stable solution $\mathbf{x}^{\text{DA}}$}
\STATE{\emph{Step 2:} Solve the inverse optimization problem given $\mathbf{x}^{\text{DA}}$. Obtain solution $\mathbf{b^*}$ that corresponds to the stability cost vector associated with the DA solution. }
\STATE{\emph{Step 3:} Solve the multi-objective optimization problem given $\mathbf{b^*}$}.
\STATE {Return $\mathbf{x}^*$. }
\end{algorithmic}
\caption{Optimization for stability and other objectives}
\label{alg:inv}
\end{algorithm}

\section{Computational Results}
This section reports computational results on the trade-offs between blocking pairs and other objectives in realistic matching problems. Section \ref{sec:sch_results} examines the Boston Public Schools (BPS) setting, showing that allowing a small number of blocking pairs yields substantial improvements in transportation outcomes. Section \ref{sec:exact_inverse_comp} compares the inverse and exact approaches on a smaller instance, demonstrating that the objective gap between the two methods is small. Finally, Section \ref{sec:residency_match} highlights the potential to match more couples to the same location with only minor violations of stability.

\subsection{School Choice} \label{sec:sch_results}
This following experiments are using synthetic data of 92 schools and 16,255 students constructed based on the BPS 2017 Transportation Challenge dataset.
We assume that the student’s preference over a school consists of three components:
travel distance, quality of the school, and random unobserved factors unique to each student. 
In particular, given $\phi_i \in [0,1]$, a student $i$'s utility from being matched to school $j$ is described by 
\begin{align} \label{stu_pref}
    p_{i,j}(\phi_i) = \frac{\phi_i}{2}d_{i,j} + \frac{\phi_i}{2} SQF_j + (1-\phi_i)X_{i,j},
\end{align}
where $SQF_j$ is the School Quality Framework score of school $j$ and $X_{i,j} \sim \text{Uniform}(0,1)$ is the random unobserved factor that affects student $i$'s preference on school $j$. 
Each student ranks between 2 to 9 schools. 
We consider 5 heterogeneous types of students corresponding to 5 unique values of $\phi_i \in [0, 0.25, 0.5, 0.75, 1.0]$. Each student is randomly selected to be in one of these 5 types. Specifically, $\phi_i= 0$ indicates that this student ranks schools purely based on some unobserved preferences independent to travel distance and quality of the school and $\phi_i = 1$ means the student ranks schools solely based on travel distances and the reported school qualities.
A student-school pair is considered admissible if the school exists in the student's preference list. This constitutes 95,076 admissible pairs. 
We further assume that schools rank students also based on travel distance and their unobserved random factor, which in reality usually includes factors such as whether the student has siblings attending the school, whether the student has special talent, etc. Given $\mu \in [0,1]$, a school $j$'s utility from admitting student $i$ is then given by 
\begin{align}\label{sch_pref}
q_{i,j}(\mu) = \mu d_{i,j} + (1-\mu) Y_{i,j},
\end{align}
where $Y_{i,j} \sim \text{Uniform}(0,1)$ is the random unobserved factor. In the following analysis, we fix $\mu = 0.75$ for all school $j$.

We first match students to schools using a student-proposing DA algorithm. DA matches 15,466 students to a school in their preference list while leaving 789 students unmatched. The average travel distance among the matched students is 3.21 miles per student. 
A key property of our algorithm is the flexibility of deciding how many students to be matched. To this end, we perform two sets of experiments. 
In the first experiment, we fix the number of students matched to be equal to the number of students matched using the DA algorithm and compare the metrics under both assignments.
Figure \ref{fig:tradeoff_DA} plots the trade-off between the percentage of blocking pairs and the percentage of reduction in average travel distance per student compared to the DA solution.
The x-axis reports the percentage of blocking pairs out of 95,076 admissible pairs. The y-axis reports the percentage of reduction in travel distance as compared to that in the DA solution.
As stability is relaxed, the gain in transportation increases. In particular, a 5\% reduction in average travel distance can be achieved with only 0.7\% blocking pairs and a 10\% reduction can be achieved with 3\% blocking pairs.\footnote{We perform a grid search on the $\lambda$'s to tune for the weights that lead to the smallest number of blocking pairs for a given travel target $T$.}
\begin{figure}[ht!]
    \centering
    \caption{Trade-off in Stability and Average Travel Distance} \includegraphics[width=0.75\linewidth]{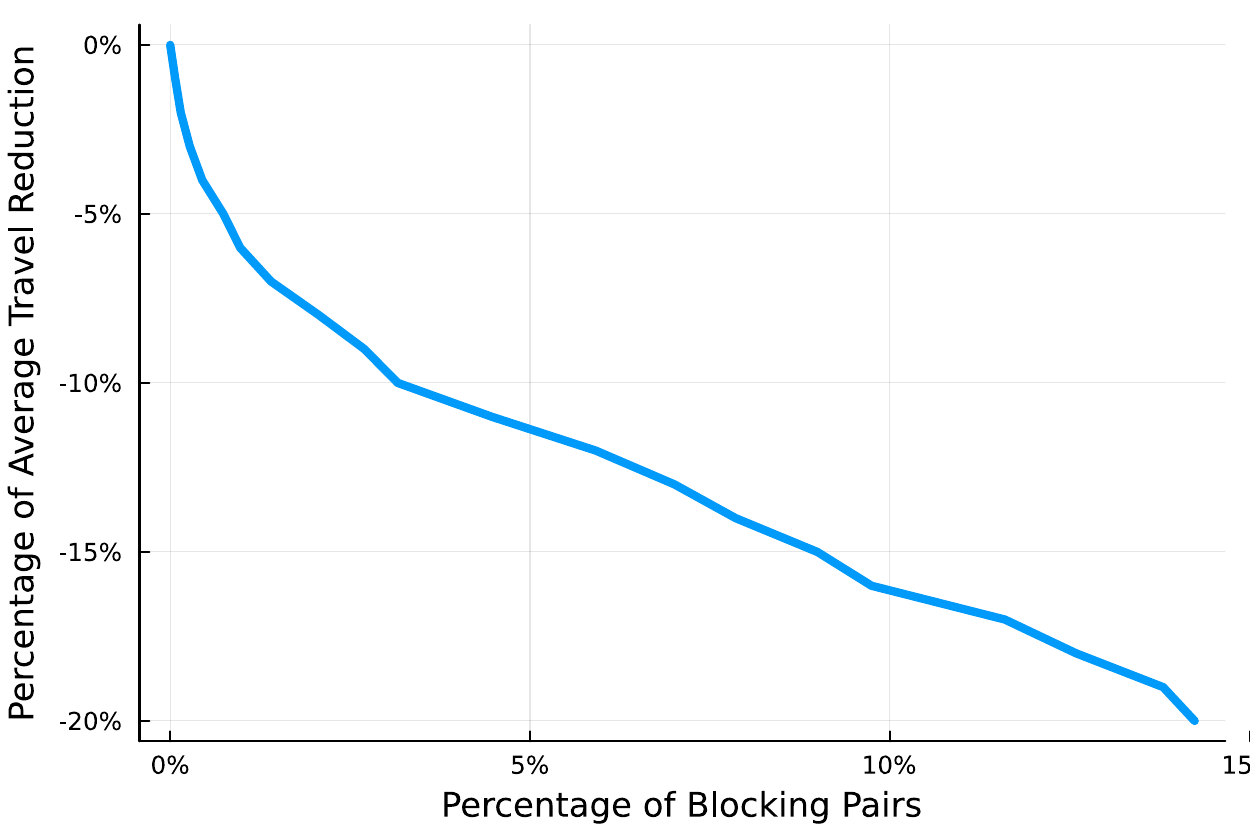}
    \label{fig:tradeoff_DA}
\end{figure}

In addition to the percentage of blocking pairs, we also directly compare the distribution of the ranks of students' matched schools in the DA solution to that in two representative points on the trade-off curve. Figure \ref{fig:stu_matched_ranks} plots the number of students receiving their first to last school choice under the DA assignment and under our assignments with 5\% and 10\% reductions of average travel. This plot indicates that the deviation in distributions is marginal when small numbers of blocking pairs are allowed. 

\begin{figure}[ht!]
    \centering
    \caption{Number of Students by Offer Rank}
    \includegraphics[width=0.75\linewidth]{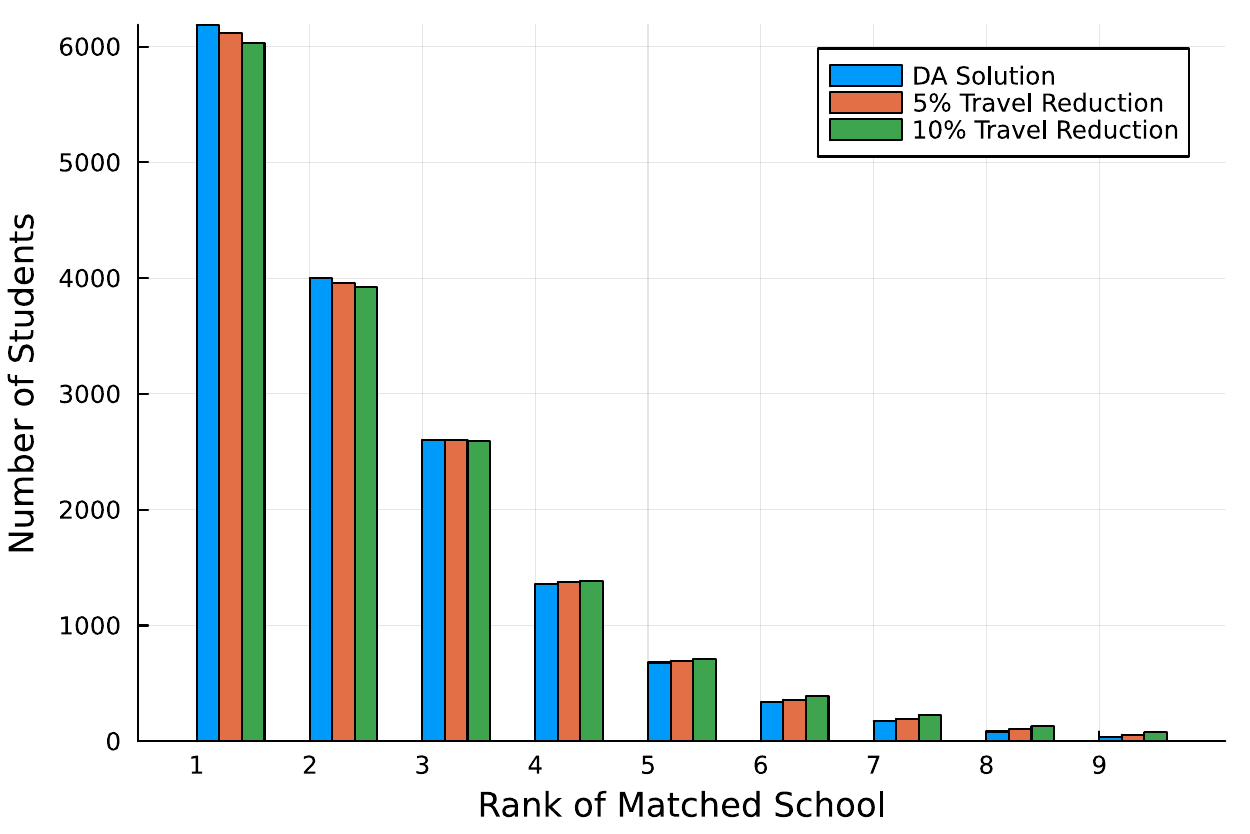}
    \label{fig:stu_matched_ranks}
\end{figure}

In the second experiment, we increase the number of students to be matched and compare the metrics as more students get matched. 
We demonstrate the flexibility of the algorithm in matching more students without increasing travel cost. In the next set of experiments, $T$ is fixed to be the original travel distance under the DA assignment while the minimum number of students matched $N$ is varied. 
Figure \ref{fig:increase_matched} plots the increase in number of blocking pairs, as the number of students matched increases. 
The DA algorithm matches 15,466 students. This indicates that this is the maximum number of students matched without constituting a blocking pair. 
Our algorithm is flexible with respect to the number of students matched. Matching 100 more students, for instance, only leads to 1.8\% of blocking pairs without increasing transportation cost. 200 more students can be assigned a seat at the expense of 3.6\% of blocking pairs and no increase in transportation cost. This suggests that with minimal number of blocking pairs allowed, we could allocate the unfilled seats in the main round optimally. 

\begin{figure}[ht!]
    \centering
    \caption{Number of Students Matched versus Stability and Average Travel Distance}
    \label{fig:increase_matched}
    \includegraphics[width=0.75\linewidth]{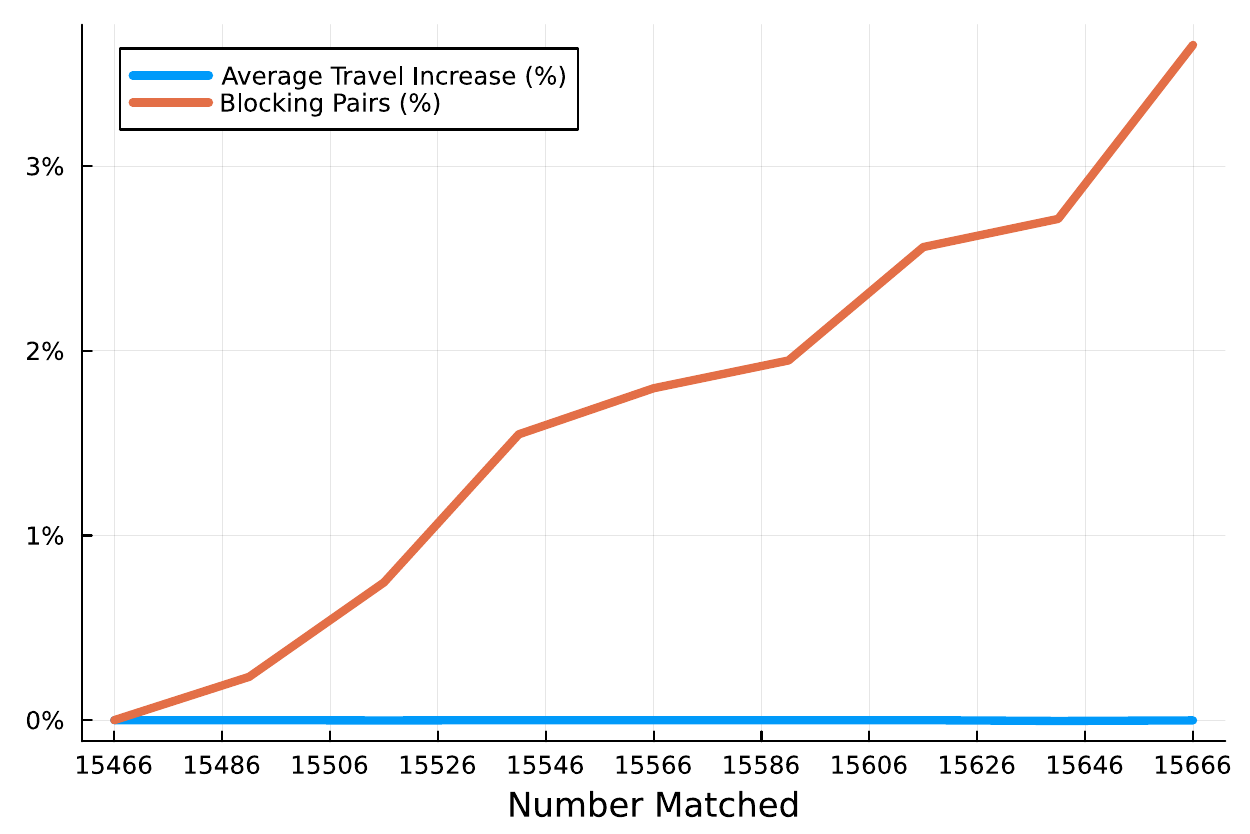}
\end{figure}

We further demonstrate that the gain in transportation cost varies as the students' decisions on how to rank the schools vary.  In the next set of experiments, we instead consider a homogeneous set of student preferences. Namely $\phi_i = \phi$ for all student $i$. We explore how the trade-offs change as student and school preference functions change. To this end, we consider 5 full sets of 16,255 students whose preferences homogeneously follow function \eqref{stu_pref} with $\phi$ values of $ 0, 0.25, 0.5, 0.75, 1.0$, respectively. Similarly, we consider 3 sets of schools whose preferences homogeneously follow function \eqref{sch_pref} with $\mu$ values of $0.5, 0.75, 1.0$. 
This yields 15 individual matching problems.  
We then run the same matching algorithm and present the trade-off curve for each of the 15 problems in Figure \ref{fig:tradeoff_DA_cf}. With 5\% of blocking pairs allowed, the gain in travel reduction varies from 3\% to 20\%. 
In particular, as higher weights are placed on travel distance when generating student preferences, the gain in average reduction becomes smaller as we relax stability. This is because if a student's rank order list consists only of schools that are close to them, then minimizing travel distance at the expense of some blocking pairs does not have a strong effect on average distance traveled. 

\begin{figure}[ht!]
    \centering
    \caption{Trade-off in Stability and Average Travel Distance as Preferences Vary}
    \includegraphics[width=0.75\linewidth]{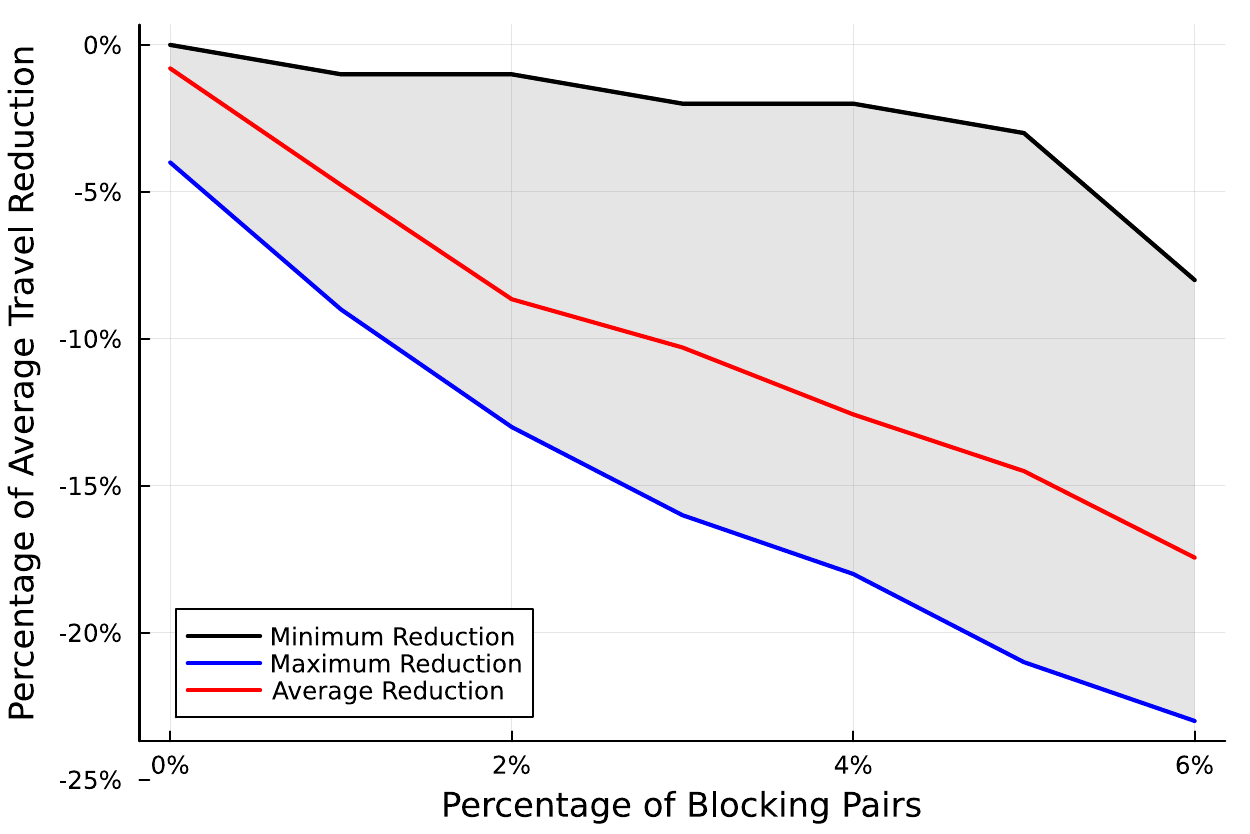}

    \label{fig:tradeoff_DA_cf}
\end{figure}

Finally, we explore different starting points for the inverse optimization problem in \eqref{const:dual} when there exists more than one stable solution. Specifically, \cite{papalexopoulos2022} proposes a MIP-based algorithm that leverages the Rural Hospitals Theorem to find another stable solution that minimizes travel distance. We then use this alternative solution to replace $\mathbf{x}^{\text{DA}}$ when solving problem \eqref{const:dual}. Figure \ref{fig:compare_DA_MIP} compares the trade-off curve using the DA stable solution versus the alternative stable solution that minimizes travel \footnote{This is not the same matching problem as the one in Figure \ref{fig:tradeoff_DA} because the stable solution for that problem is unique. }. The trade-offs are indistinguishable between using the two starting points. This is consistent with the previous finding that optimizing for
travel distance while maintaining strict stability yields minimal improvement.

\begin{figure}[ht!]
    \centering
    \caption{Different Stable Solutions as Starting Points}
    \includegraphics[width=0.75\linewidth]{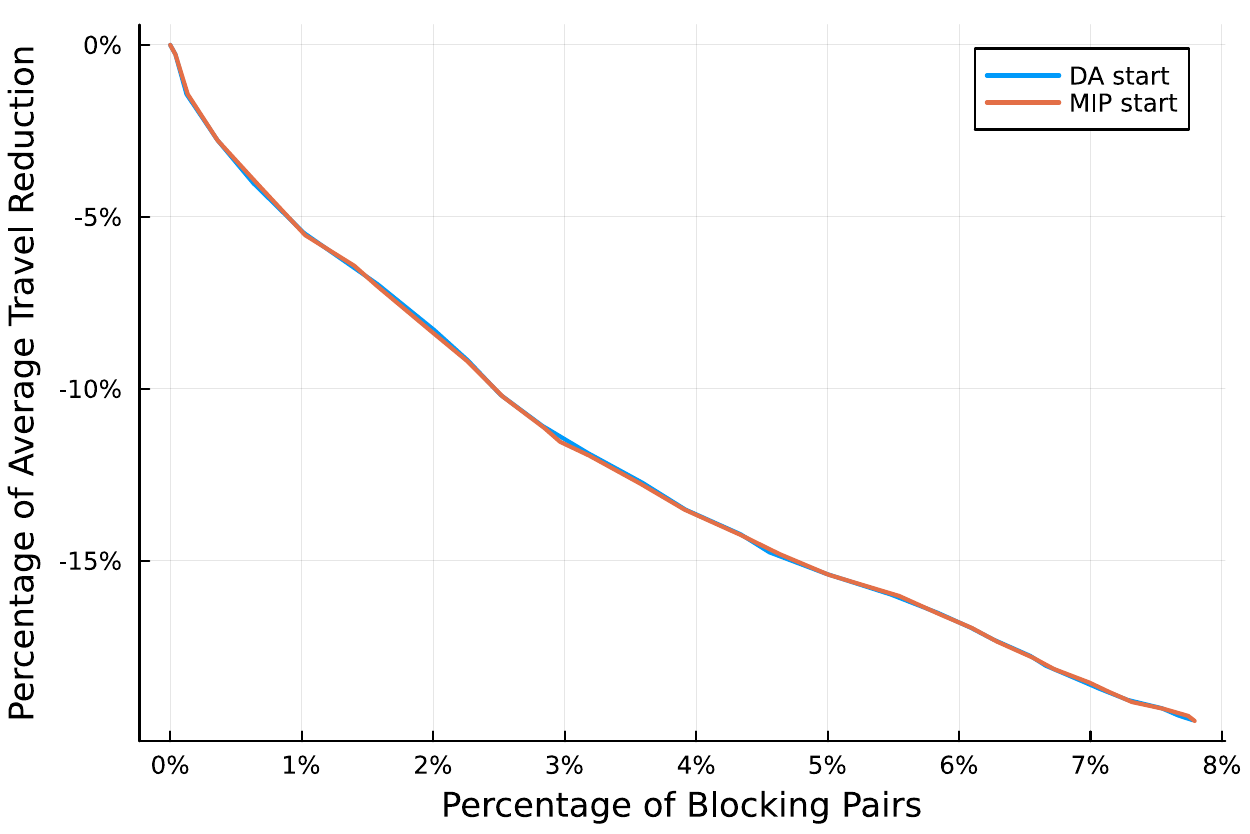}
    \label{fig:compare_DA_MIP}
\end{figure}

\subsection{Comparison with Exact Method}\label{sec:exact_inverse_comp}
In this section, we compare the reductions in student average travel using the exact method and the inverse method for different numbers of blocking pairs allowed. 
Recall that in the \textbf{exact method}, the problem is formulated as minimizing total travel distance subject to feasibility constraints and no more than $B$ blocking pairs, namely 
\begin{align*}
    \min_{\mathbf{x,y,z,r,w,a} } \quad & \sum_{(i,j) \in \mathcal{S}} d_{i,j} x_{i,j} \\
    \text{s.t.} \quad & \text{Constraints } \eqref{const:stu_unmatched} - \eqref{const:blocking_number} \quad &\text{ (Stability constraints)},\\
    & \text{Constraints } \eqref{const:unique_sch} - \eqref{const:min_matched} \quad &\text{ (Feasibility constraints)}, \\
    & \textbf{x,r,w,a} \in \{0,1\}^{|\mathcal{S}|},  \\
    & \textbf{y} \in \{0,1\}^{|\mathcal{I}|}, \\
    & \textbf{z} \in \{0,1\}^{|\mathcal{J}|}.
\end{align*}
Since the exact method does not scale to problem sizes in Section \ref{sec:sch_results}, we perform the comparison on a reduced problem of 200 students and 7 schools using a subset of the students and schools.
In the following experiment, we first use the exact method and obtain the corresponding travel reduction when we incrementally allow 0\% to 20\% blocking pairs. We then implement the inverse method to match the same amount of travel reduction and compare the number of blocking pairs created. Figure \ref{fig:compare_method} compares the trade-offs between travel reduction and blocking pairs between the two methods. 
Admittedly, the exact method is able to achieve a smaller percentage of blocking pairs for the same amount of travel reduction.
However, the inverse method solves in seconds for each set of weights and in minutes to a realistic problem sizes as reported in Section \ref{sec:sch_results}, while the exact method fails to find an optimal solution or even a feasible solution in hours. 
Moreover, for the same amount of travel reduction, the gap in percent of blocking pairs between the two methods is no more than 1\%. The gap begins to shrink even further when more than 10\% of blocking pairs are allowed. 
\begin{figure}[ht!]
    \centering
    \caption{Comparison between Exact and Inverse Methods}
    \includegraphics[width=0.745\linewidth]{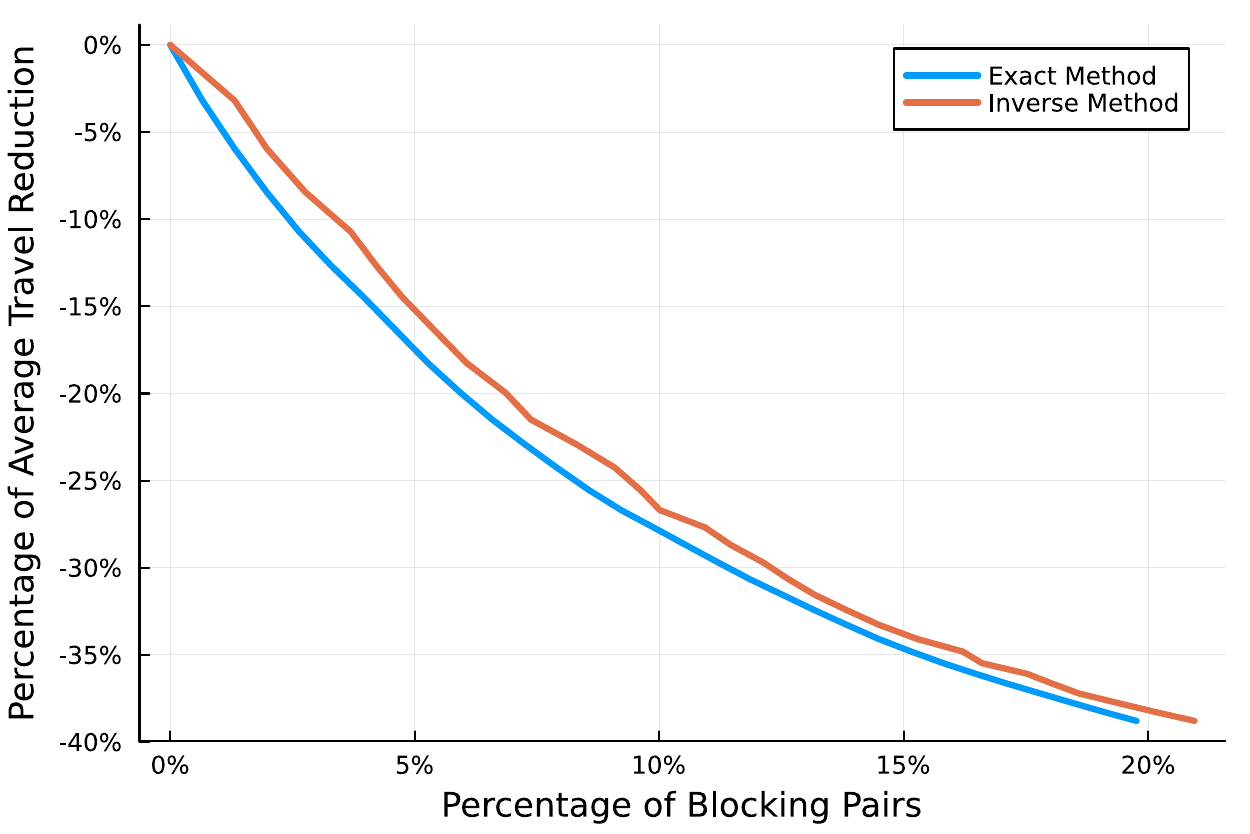}
    \label{fig:compare_method}
\end{figure}

Appendix \ref{appendix:gap_bound} provides a formal exponential tail bound on the gap between the number of blocking pairs under the inverse and exact methods. In particular, Theorem \ref{thm:bp_tailbound} shows that the probability of the inverse method producing substantially more blocking pairs than the exact method is exponentially small.

\subsection{Residency Match}\label{sec:residency_match}
We consider a synthetic residency match with 92 hospitals and 16,255 residents of which 2,000 are applying with a partner. Namely, there are 1,000 pairs of couples, accounting for roughly 12\% of the applicants. 
Figure \ref{fig:tradeoff_DA_residency} plots the trade-off between the percentage of blocking pairs and the percentage of couples matched to the same location.
As a benchmark, the match obtained using the resident-optimal DA algorithm matches 549 out of the 1,000 pairs to the same location. This corresponds to the left corner of the plot with 0\% of blocking pairs. 
Allowing 5\% of blocking pairs matches nearly 60\% of all possible couples to the same location. It is also noteworthy that this accounts for a 9\% increase compared to the numbers of couples matched to the same location in the DA solution.
A more aggressive relaxation of stability leads to 10\% of blocking pairs in trade-off to 67\% of all possible pairs matched to the same location, accounting for 23\% improvement compared to the DA solution. 

\begin{figure}[ht!]
    \caption{Trade-off in Stability and Couples Matched to Same Location}
    \centering
    \includegraphics[width=0.75\linewidth]{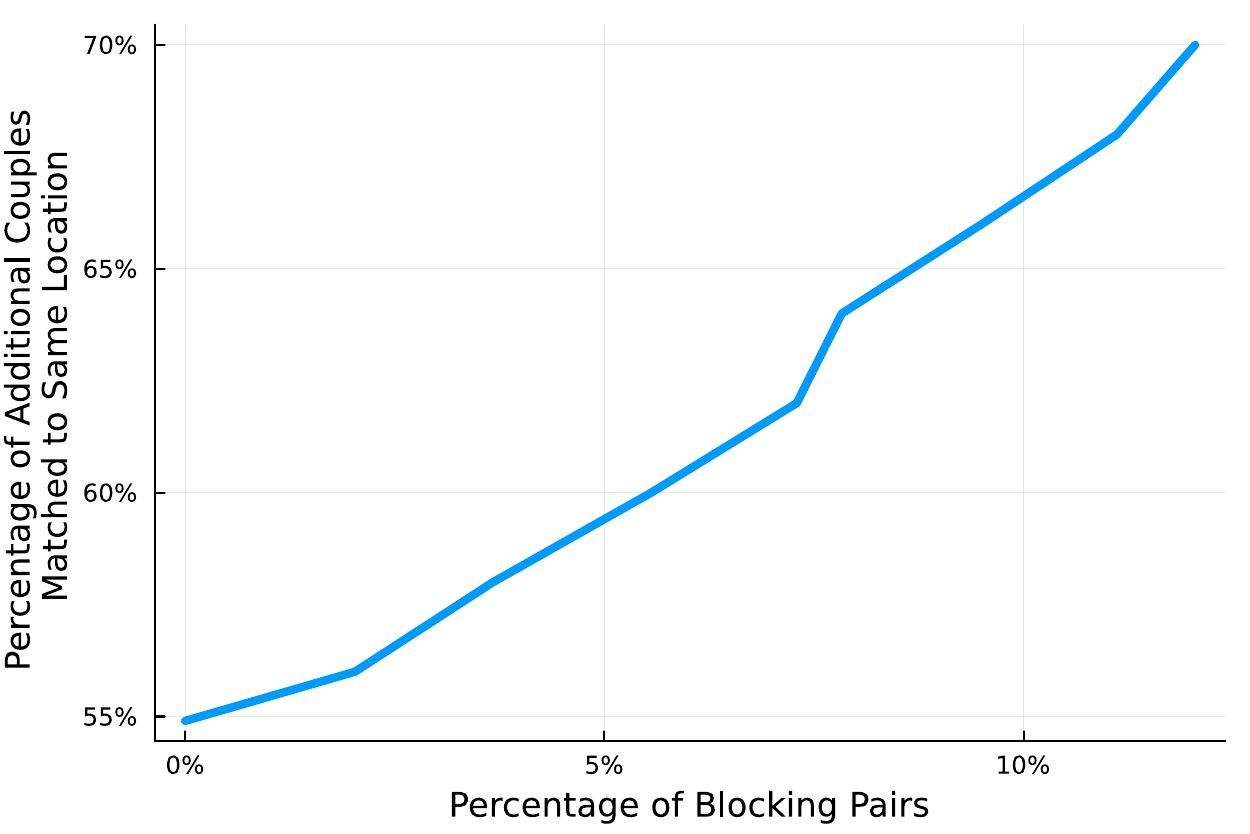}
    \label{fig:tradeoff_DA_residency}
\end{figure}

Figure \ref{fig:resident_matched_ranks} further plots the number of residents receiving their first to last choice under the DA assignment and under our assignments with 5\% and 15\% of blocking pairs. Similar to the school choice problem, the distribution in applicants' offer ranks remains similar when stability is relaxed. 

\begin{figure}[ht!]
    \centering
    \caption{Number of Residents by Offer Rank}
    \includegraphics[width=0.75\linewidth]{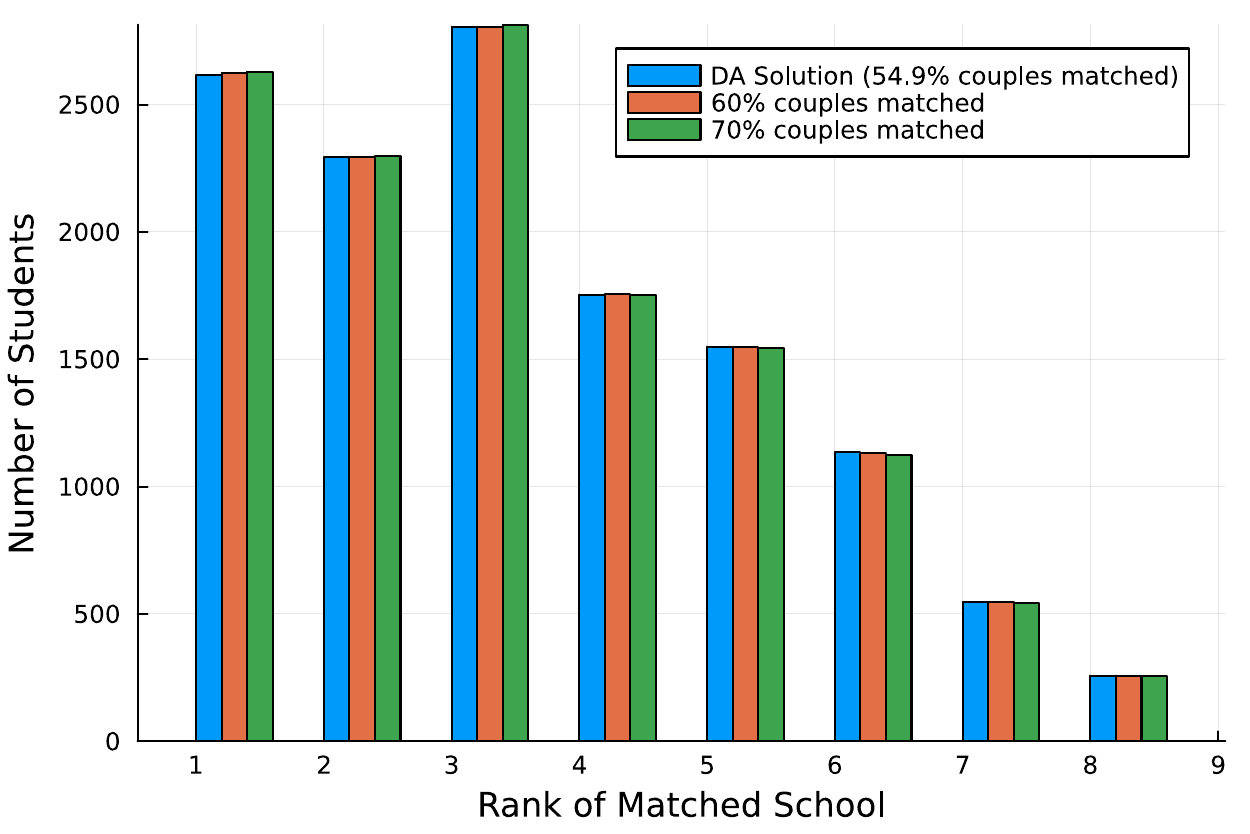}
    \label{fig:resident_matched_ranks}
\end{figure}

\section{Conclusion}
In this work, we develop a novel optimization-based algorithm for solving large-scale many-to-one matching with multiple objectives, allowing for controlled relaxations of stability in the context of school choice and residency match. Computational experiments suggest that permitting as low as 5\% of blocking pairs can yield substantial improvements in other objectives. In particular, average travel distance in Boston Public Schools can be reduced by up to 20\% while maintaining a distribution of student offer ranks comparable to that under Deferred-Acceptance. Similarly, residency match outcomes show a 10\% increase in couples assigned to the same location. Moreover, relaxing stability offers flexibility in shaping outcomes, such as increasing the number of students matched to schools on their preference lists, therefore boosting first-round enrollment in public schools. 

Overall, these results provide empirical evidence that strict stability may be too rigid in practice. Allowing for a small, carefully bounded level of instability can meaningfully improve welfare in centralized matching markets. Thus, in response to the guiding question of this paper, our findings suggest that relaxing stability is not only viable, but in many contexts, desirable.

\section{Acknowledgment}
We thank Theodore Papalexopoulos and Felipe Cordera for helpful comments.

\appendix 
\section{Appendix}
\subsection{Bound on Blocking Pairs under Inverse Method} \label{appendix:gap_bound}

Given a solution to the exact problem, denoted by ${\bf x}^{\text{exact}}$, we can solve the same inverse optimization in Equations \eqref{const:inv_objective} - \eqref{const:normalization} to obtain a cost vector ${\bf b}^{\text{true}}$ that makes ${\bf x}^{\text{exact}}$ optimal. 
Let \( {\bf b}^{\text{inv}} \) denote a cost vector recovered via inverse optimization from a stable solution \( {\bf x}^{\text{DA}} \). 

\begin{assumption}[Sub-Gaussian Inverse Error]\label{assum:subgaussian}
Let \( \Delta {\bf b} := {\bf b}^{\textnormal{inv}} - {\bf b}^{\textnormal{true}} \in\mathbb{R}^{|\mathcal{S}|} \) denote the deviation between the inverse-optimized cost vector associated with the stable solution and the latent true cost vector associated with ${\bf x}^{\text{exact}}$. We assume \( \Delta {\bf b} \) is an i.i.d. sub-Gaussian random vector with zero mean and variance proxy \( \sigma^2 > 0\) .
Then, 
\[
\mathbb{P}\left( \|\Delta {\bf b}\|_2 \geq \sigma \sqrt{|\mathcal{S}|} + t \right)
\leq \exp\left( -\frac{t^2}{2\sigma^2} \right), \quad \text{for all } t > 0.
\]
\end{assumption}
This sub-Gaussian assumption reflects that the inverse-optimized cost vector recovered from the stable solution concentrates around the true latent cost vector. This assumption is most reasonable when the blocking pair budget \( B \) is small, so that the solution \( {\bf x}^{\text{exact}} \) is close in structure to the stable solution \( {\bf x}^{\text{DA}} \). As \( B \) increases, the objectives optimized by \( {\bf b}^{\text{inv}} \) and \( {\bf b}^{\text{true}} \) may diverge significantly, weakening the validity of this assumption.

\begin{lemma}[Lipschitz Continuity of Blocking Pairs]\label{lem:lipschitz-bp-x}
Let \( {\bf x}, {\bf x}' \in \mathcal{F} \subseteq \{0,1\}^{|\mathcal{S}|} \) be two feasible assignment vectors, and let \( \textnormal{BP}({\bf x}) \) denote the number of blocking pairs under assignment \( {\bf x} \). Let $p_{\max} = \max_i |\mathcal{P}(i)|$ denote the maximum number of schools ranked by a student and $C_{\max} = \max_j C_j$ denote the maximum capacity across all schools.
Then \( \textnormal{BP}({\bf x}) \) is Lipschitz continuous with respect to the assignment under the \( \ell_1 \)-norm, with Lipschitz constant \( \frac{1}{2}(p_{\max} + 2C_{\max}) \). Namely, 
\[
\left| \textnormal{BP}({\bf x}) - \textnormal{BP}({\bf x}') \right|
\leq  \frac{1}{2}(p_{\max} + 2C_{\max}) \cdot \| {\bf x} - {\bf x}' \|_1.
\]
\end{lemma}
\begin{proof}
Let \( {\bf x}, {\bf x} \) be two feasible matchings. Suppose they differ in the assignment of a set of students \( A \subseteq \mathcal{I} \), where for each student \( i \in A \), their assignment in \( {\bf x} \) and \( {\bf x}' \) corresponds to different schools $j$ and $k$. Then, for each one of these students, 
$$|x_{i,j} - x'_{i,j}| = |x_{i,k} - x'_{i,k}| = 1.$$
Namely, each reassignment of a student \( i \) changes exactly two entries in the binary assignment vector \( {\bf x} \). Thus, for each such change of assignment, the \( \ell_1 \) difference increases by 2. Hence,
\[
|A| = \frac{1}{2} \| {\bf x} - {\bf x}' \|_1.
\]

We then consider the impact of reassigning a single student \( i \) on the number of blocking pairs:
\begin{itemize}
    \item The student \( i \), whose assignment changes, can potentially form up to \( |\mathcal{P}(i)| \leq p_{\max} \) new or removed blocking pairs with schools in their preference list.
    \item The school \( j \) that loses student \( i \) may have its priority list updated, possibly changing blocking pair status with up to \( C_{j} \leq C_{\max} \) students.
    \item Similarly, the school \( k \) that gains student \( i \) may affect up to \( C_{k} \leq C_{\max} \) other blocking pairs.
\end{itemize}
Thus, the total number of blocking pairs that can change due to reassigning student \( i \) is at most:
\[
|\mathcal{P}(i)|  + C_{j} + C_{k} \leq  p_{\max} + 2C_{\max}.
\]
Therefore, 
\[
| \text{BP}({\bf x}) - \text{BP}({\bf x}') | \leq \sum_{i \in A} ( p_{\max} + 2C_{\max}) = ( p_{\max}  + 2C_{\max}) \cdot |A| = \frac{1}{2} ( p_{\max} + 2C_{\max}) \cdot \| {\bf x} - {\bf x}' \|_1.
\]
\end{proof}

\begin{assumption}[Lipschitz Continuity of Assignment]\label{assum:lipschitz}
There exists a constant \( L > 0 \) such that for any two cost vectors \( {\bf b}, {\bf b}' \in \mathbb{R}^{|\mathcal{S}|} \), their corresponding optimal solutions \( {\bf x}({\bf b}), {\bf x}({\bf b}') \) satisfy
\[
\|{\bf x}({\bf b}) - {\bf x}({\bf b}')\|_1 \leq L_2 \cdot \|{\bf b} - {\bf b}'\|_2.
\]
\end{assumption}

\begin{theorem}[Exponential Tail Bound on Blocking Pairs] \label{thm:bp_tailbound}
Suppose \( \varepsilon > L \sigma \sqrt{|\mathcal{S}|} \) and Assumptions \ref{assum:subgaussian} and \ref{assum:lipschitz} hold. 
Then the number of blocking pairs of the inverse-optimized solution satisfies
\[
\mathbb{P}\left( \textnormal{BP}({\bf x}^{\textnormal{inv}}) \geq \textnormal{BP}({\bf x}^{\textnormal{exact}}) + \varepsilon \right)
\leq \exp\left( -\frac{(\varepsilon - L \sigma \sqrt{|\mathcal{S}|})^2}{2 L^2 \sigma^2} \right), 
\]
where $ L = \frac{1}{2}(p_{\max} + 2C_{\max}) L_2$.
\end{theorem}

\begin{proof}
By Lemma \ref{lem:lipschitz-bp-x} and Assumption \ref{assum:lipschitz}, 
\[
|\text{BP}({\bf x}({\bf b})) - \text{BP}({\bf x}({\bf b}'))| \leq \frac{1}{2}(p_{\max} + 2C_{\max}) L_2 \cdot \|{\bf b} - {\bf b}'\|_2.
\]
Let \( \Delta {\bf b} := {\bf b}^{\text{inv}} - {\bf b}^{\text{true}} \in \mathbb{R}^{|\mathcal{S}|} \).
By Assumption \ref{assum:lipschitz}, we have:
\[
| \text{BP}({\bf x}^{\text{inv}}) - \text{BP}({\bf x}^{\text{exact}}) | \leq L \cdot \| \Delta {\bf b} \|_2.
\]
Also note that for a given travel $T$, \(\text{BP}({\bf x}^{\text{inv}}) \geq \text{BP}({\bf x}^{\text{exact}}) \). Then, 
\[
\text{BP}({\bf x}^{\text{inv}}) \geq \text{BP}({\bf x}^{\text{exact}}) + \varepsilon
\implies
L \cdot \| \Delta {\bf b} \|_2 \geq \varepsilon 
\iff  \| \Delta {\bf b} \|_2 \geq  \frac{\varepsilon }{L}
\]
By Assumption \ref{assum:subgaussian}, the following inequality holds for all $t >0$:
\[
\mathbb{P}\left( \| \Delta {\bf b} \|_2 \geq \sigma \sqrt{|\mathcal{S}|} + t \right)
\leq \exp\left( -\frac{t^2}{2\sigma^2} \right). 
\]
Let \( t := \frac{\varepsilon}{L} - \sigma \sqrt{|\mathcal{S}|} \). Since \( \varepsilon > L \sigma \sqrt{|\mathcal{S}|} \), then \( t > 0 \). Then it follows that 
\[
\mathbb{P}\left( \text{BP}({\bf x}^{\text{inv}}) \geq \text{BP}({\bf x}^{\text{exact}}) + \varepsilon \right)
\leq \mathbb{P}\left( \| \Delta {\bf b} \|_2 \geq \frac{\varepsilon}{L} \right)
= \mathbb{P}\left( \| \Delta {\bf b} \|_2 \geq \sigma \sqrt{|\mathcal{S}|} + t \right)
\leq \exp\left( -\frac{t^2}{2\sigma^2} \right).
\]
Substituting back yields
\[
\mathbb{P}\left( \text{BP}({\bf x}^{\text{inv}}) \geq \text{BP}({\bf x}^{\text{exact}}) + \varepsilon \right)
\leq
\exp\left( -\frac{\left( \frac{\varepsilon}{L} - \sigma \sqrt{|\mathcal{S}|} \right)^2}{2 \sigma^2} \right)
= \exp\left( -\frac{(\varepsilon - L \sigma \sqrt{|\mathcal{S}|})^2}{2 L^2 \sigma^2} \right),
\]
as desired.
\end{proof}

\bibliographystyle{ecta}
\bibliography{citations.bib}
\end{document}